\documentclass[11pt]{article}
\usepackage{indentfirst}
\usepackage{amsmath, amscd, amssymb, amsthm, verbatim}

\allowdisplaybreaks

\newtheorem{theorem}{Theorem}[section]
\newtheorem{corollary}[theorem]{Corollary}
\newtheorem{lemma}[theorem]{Lemma}

\newtheorem{remark}[theorem]{Remark}

\begin{document}
\title{Two nonlocal inverse curvature flows of convex closed plane curves}

\author{Zezhen Sun \footnote{School of Mathematical Sciences, East China Normal University, Shanghai 200241, China. E-mail address: \texttt{52205500017@stu.ecnu.edu.cn}}
}

\maketitle
\begin{abstract}
  \noindent
  In this paper we introduce two $1/\kappa^{n}$-type ($n\ge1$) curvature flows for closed convex planar curves. Along the flows the length of the curve is decreasing while the enclosed area is increasing. And finally, the evolving curves converge smoothly to a finite circle if they do not develop singularity during the evolution process.
\end{abstract}

\section{Introduction}
The curvature flow of plane curves, arising in many application fields, such as phase transitions, image processing and smoothing, etc., has received a lot of attention during the past
several decades. In general, the evolution equation has the form
\begin{equation}\label{fl}
  \left\{\begin{aligned}
    \frac{\partial X(u, t)}{\partial t}&=f(\kappa(u,t))N_{in}(u,t),\\
    X(u, 0)&=X_{0}(u),\quad u\in S^{1},
  \end{aligned}
  \right.
  \end{equation}
where $X_{0}(u)\subset\mathbb{R}^{2}$ is a given smooth closed curve, parameterized by $u\in S^{1}$, and $X(u, t):S^{1}\times[0,T)\rightarrow\mathbb{R}^{2}$ is a family of curves moving along its inward normal direction $
N_{in}(u,t)$ with given speed function $f(\kappa(u,t))$, which is a strictly increasing (parabolicity) function
of the curvature $\kappa(u,t)$ of $X(u, t)$.

When $f(\kappa(u,t))=\kappa(u,t)$  is the famous curve shortening flow, which has been
studied intensively by a great number of authors for various conditions on the initial curve $X_{0}$. One can see the book \cite{Chou-Zhu} for literature. In particular, we mention the papers \cite{Gage83,Gage85,Gage-Hamilton,Grayson,Grayson87}. Another class of interesting curvature flow is the so-called nonlocal curvature flow, the evolution equation of which takes the form
\begin{equation}\label{fl}
  \left\{\begin{aligned}
    \frac{\partial X(u, t)}{\partial t}&=F(\kappa(u,t)-\lambda(t))N_{in}(u,t),\\
    X(u, 0)&=X_{0}(u),\quad u\in S^{1},
  \end{aligned}
  \right.
  \end{equation}
where $X_{0}$ is a convex simple closed curve,
$F(\kappa(u,t))$ is a given function of the curvature satisfying the parabolic condition $F^{'}(z)>0$
for all $z$ in its domain, and $\lambda(t)$ is a function of time which depends on certain global
(nonlocal) quantities of $X (\cdot, t)$, say its length $L(t)$, enclosed area $A(t)$, or other possible
global quantities like the integral of curvature over the entire curve in certain ways. In \cite{Tsai-Wang},
Tsai-Wang  considered the following $\kappa^{n}$-type nonlocal flows
\begin{align}
F(\kappa(u,t)-\lambda(t))&=\kappa^{n}-\frac{1}{2\pi}\int_{X (\cdot, t)}\kappa^{n+1}ds,\quad n>0\quad(LP),\notag\\
F(\kappa(u,t)-\lambda(t))&=\kappa^{n}-\frac{1}{L(t)}\int_{X (\cdot, t)}\kappa^{n}ds,\quad n>0\quad(AP),\notag\\
\label{ldai1}F(\kappa(u,t)-\lambda(t))&=\kappa^{n}-\frac{2A(t)}{L^{2}(t)}\int_{X (\cdot, t)}\kappa^{n+1}ds,\quad n\ge1\quad(LD\,and\,AI),\\
\label{ldai2}F(\kappa(u,t)-\lambda(t))&=\kappa^{n}-\frac{L(t)}{4\pi A(t)}\int_{X (\cdot, t)}\kappa^{n+1}ds,\quad n\ge1\quad(LD\,and\,AI).
\end{align}
Here $s$ is the arc length parameter of $X (\cdot, t)$ and the constant $n>0$ is arbitrary. The
abbreviation $LP (AP,LD\, and\, AI)$ indicates that the flow is length-preserving (area-preserving,length-decreasing and area-increasing). Later, Gao-Pan-Tsai studied the following  $1/\kappa^{n}$-type flows (see \cite{gpt1,gpt2})
\begin{align}
\label{lp}F(\kappa(u,t)-\lambda(t))&=\frac{1}{2\pi}\int_{X (\cdot, t)}\kappa^{1-n}ds-\frac{1}{\kappa^{n}},\quad n>0\quad(LP),\\
\label{ap}F(\kappa(u,t)-\lambda(t))&=\frac{1}{L(t)}\int_{X (\cdot, t)}\kappa^{-n}ds-\frac{1}{\kappa^{n}},\quad n>0\quad(AP).
\end{align}
Inspired by their works, we now investigate two new $1/\kappa^{n}$-type nonlocal curvature flows for convex curves with the speed function $F(\kappa)-\lambda(t)$ given by
\begin{equation}\label{flow1}
F(\kappa)-\lambda(t)=\frac{L(t)}{2L^{2}(t)-4\pi A(t)}\int_{X (\cdot, t)}\kappa^{-n}ds-\frac{1}{\kappa^{n}},\quad n\ge1,
\end{equation}
or
\begin{equation}\label{flow2}
F(\kappa)-\lambda(t)=\frac{L^{2}(t)-2\pi A(t)}{\pi L^{2}(t)}\int_{X (\cdot, t)}\kappa^{1-n}ds-\frac{1}{\kappa^{n}},\quad n\ge1.
\end{equation}
Both flows have the common feature that the length $L(t)$ is decreasing and the area $A(t)$ is increasing.

Since the initial curve $X_{0}$ has positive curvature $\kappa_{0}(u)$ for all $u\in S^{1}$ and the curvature term $F(\kappa(u,t))=-\kappa^{-n}$ in the speed function of \eqref{fl} is strictly increasing on its domain $\kappa\in(0,\infty)$, the flow \eqref{flow1}(or \eqref{flow2}) is parabolic. Similar to the discussions
in \cite{Mao-Pan-Wang,Jiang-Pan} (using Leray-Schauder's fixed point theory), or in \cite{GZ} (using the linearization method), or in many other nonlocal flow papers, there is a unique smooth convex solution of \eqref{flow1}(or \eqref{flow2}) defined on $S^{1}\times[0,T)$ for short time $T>0$. Therefore, we have short-time existence of a convex solution to \eqref{flow1}(or \eqref{flow2}).

The purpose of this paper is to study the asymptotic behavior of the flow \eqref{flow1}(or \eqref{flow2}). Similar to what happens in the $1/\kappa^{n}$-type length-preserving flow \eqref{lp} and $1/\kappa^{n}$-type area-preserving flow \eqref{ap}, singularity (curvature blow-up to $\infty$) may occur under the flow \eqref{flow1}(or \eqref{flow2}) for some initial convex closed curves. Thus, the optimal result we can obtain is the following,which is the main result in this paper.
\begin{theorem}\label{them}
Assume $n\ge1$ and $X_{0}(u)$ is a smooth convex closed curve.Consider the flow \eqref{flow1}(or \eqref{flow2}) and assume that the curvature $\kappa$ will not blow up to $\infty$ in any finite time during the evolution process. Then the flow exists for all time $t\in[0,\infty)$.  The length $L(t)$ is decreasing while the
enclosed area $A(t)$ is increasing. When $t\to\infty$, the evolving curve converges to a finite circle in $C^{\infty}$ topology.
\end{theorem}
\section{The proof of Theorem \ref{them}}
\subsection{Some basic evolution equations}
Before we start to prove the Theorem \eqref{them}, we list some basic evolution equations. It is well-known that if $X(\cdot, t):S^{1}\times[0,T)\rightarrow\mathbb{R}^{2}$ is a family of evolving simple closed curves, its length $L(t)$ and enclosed area $A(t)$ satisfy the following equations(see\cite{Chou-Zhu,Sun})
\begin{equation}\label{LA}
\frac{dL(t)}{dt}=-\int_{X(\cdot,t)}\langle V,\kappa N_{in}\rangle ds,\quad
\frac{dA(t)}{dt}=-\int_{X(\cdot,t)}\langle V,N_{in}\rangle ds,
\end{equation}
where $V=\frac{\partial X}{\partial t}$ is the velocity vector of $X$ and $\langle,\rangle$ is the inner product in $\mathbb{R}^{2}$.

\begin{lemma}\label{llaa}
(The monotonicity of length and area)Assume $X(u, t):S^{1}\times[0,T)\rightarrow\mathbb{R}^{2}$ is a smooth convex solution of the flow \eqref{flow1}(or \eqref{flow2}). Then we have
\begin{equation}
\frac{dL(t)}{dt}\le0,
\quad\frac{dA(t)}{dt}\ge0,\quad\forall t\in[0,T).
\end{equation}
\end{lemma}
\begin{proof}
For the flow \eqref{flow1}, we have
\begin{equation}\label{la1}
  \left\{\begin{aligned}
    \frac{dL(t)}{dt}&=-\int_{X(\cdot,t)}\bigg(\frac{L(t)}{2L^{2}(t)-4\pi A(t)}\int_{X (\cdot, t)}\frac{1}{\kappa^{n}}ds-\frac{1}{\kappa^{n}}\bigg)\kappa ds\notag\\
    &=-\frac{\pi L(t)}{L^{2}(t)-2\pi A(t)}\int_{X(\cdot,t)}\frac{1}{\kappa^{n}}ds+\int_{X(\cdot,t)}\frac{1}{\kappa^{n-1}}ds\\
    &=-\frac{\pi L(t)}{L^{2}(t)-2\pi A(t)}\int^{2\pi}_{0}\frac{1}{\kappa^{n+1}(\theta,t)}d\theta+\int^{2\pi}_{0}\frac{1}{\kappa^{n}(\theta,t)}d\theta,\\
    \frac{dA(t)}{dt}&=-\int_{X(\cdot,t)}\bigg(\frac{L(t)}{2L^{2}(t)-4\pi A(t)}\int_{X (\cdot, t)}\frac{1}{\kappa^{n}}ds-\frac{1}{\kappa^{n}}\bigg)ds\notag\\
    &=-\frac{L^{2}(t)}{2L^{2}(t)-4\pi A(t)}\int_{X(\cdot,t)}\frac{1}{\kappa^{n}}ds+\int_{X(\cdot,t)}\frac{1}{\kappa^{n}}ds\\
    &=-\frac{L^{2}(t)}{2L^{2}(t)-4\pi A(t)}\int^{2\pi}_{0}\frac{1}{\kappa^{n+1}(\theta,t)}d\theta+\int^{2\pi}_{0}\frac{1}{\kappa^{n+1}(\theta,t)}d\theta.
  \end{aligned}
  \right.
  \end{equation}
Similarly, for the flow \eqref{flow2}, we have
\begin{equation}\label{la2}
  \left\{\begin{aligned}
    \frac{dL(t)}{dt}&=-\frac{2L^{2}(t)-2\pi A(t)}{L^{2}}\int^{2\pi}_{0}\frac{1}{\kappa^{n}(\theta,t)}d\theta+\int^{2\pi}_{0}\frac{1}{\kappa^{n}(\theta,t)}d\theta,\\
    \frac{dA(t)}{dt}&=-\frac{L^{2}(t)-2\pi A(t)}{\pi L(t)}\int^{2\pi}_{0}\frac{1}{\kappa^{n}(\theta,t)}d\theta+\int^{2\pi}_{0}\frac{1}{\kappa^{n+1}(\theta,t)}d\theta.
  \end{aligned}
  \right.
\end{equation}
By combining the the classical isoperimetric inequality $L^{2}(t)\ge4\pi A(t)$ , the Lin-Tsai' s inequality (see \cite{Lin-Tsai2012})
\begin{equation}\label{LT}
\int^{2\pi}_{0}\frac{1}{\kappa^{n}(\theta,t)}d\theta\le\frac{\pi L(t)}{L^{2}(t)-2\pi A(t)}\int^{2\pi}_{0}\frac{1}{\kappa^{n+1}(\theta,t)}d\theta,\quad\forall n\ge1,
\end{equation}
and the following H$\ddot{o}$lder inequality
\begin{equation}\label{hold}
L(t)=\int^{2\pi}_{0}\frac{1}{\kappa(\theta,t)}d\theta\le\bigg(\int^{2\pi}_{0}\frac{1}{\kappa^{n+1}(\theta,t)}d\theta\bigg)
^{\frac{1}{n+1}}\bigg(\int^{2\pi}_{0}d\theta\bigg)^{\frac{n}{n+1}}.
\end{equation}
we can complete the proof.
\end{proof}

As a consequence of Lemma \eqref{llaa}, we can easily get
\begin{lemma}\label{lajie}
Assume $X(u, t):S^{1}\times[0,T)\rightarrow\mathbb{R}^{2}$ is a smooth convex solution of the flow \eqref{flow1}(or \eqref{flow2}). Then
\begin{equation}\label{la}
\sqrt{4\pi A(0)}\le L(t)\le L(0)\quad and \quad A(0)\le A(t)\le\frac{L^{2}(0)}{4\pi},\quad\forall t\in[0,T),
\end{equation}
and the isoperimetric ratio $\frac{L^{2}(t)}{4\pi A(t)}$ is decreasing during the evolution process unless the initial curve $X_{0}$ is a circle (which is an equilibrium solution of the flow \eqref{flow1} or \eqref{flow2}).
\end{lemma}

Let $X(u,t)$ be a convex solution of the flow \eqref{flow1}(or \eqref{flow2}) on $S^{1}\times[0,T)$. By convexity, we can use the the outward normal angle $\theta\in[0,2\pi]$ to express the curvature $\kappa$ and other geometric quantities.The evolution equation of $\kappa$ is given by
\begin{equation}\label{kt}
  \left\{\begin{aligned}
    \frac{\partial\kappa}{\partial t}(\theta,t)&=\kappa^{2}(\theta,t)\bigg(\big(-\frac{1}{\kappa^{n}(\theta,t)}\big)_{\theta\theta}
+\lambda(t)-\frac{1}{\kappa^{n}(\theta,t)}\big)\bigg),\\
    \kappa(\theta,0)&=\kappa_{0}(\theta)>0,\quad(\theta,t)\in S^{1}\times[0,T),
  \end{aligned}
  \right.
  \end{equation}
  where $\kappa_{0}(\theta)>0$ is the curvature of $X_{0}$ and
  \begin{equation*}
  \lambda(t)=\frac{L(t)}{2L^{2}(t)-4\pi A(t)}\int_{X (\cdot, t)}\kappa^{-n}ds=\frac{L(t)}{2L^{2}(t)-4\pi A(t)}
  \int^{2\pi}_{0}\frac{1}{\kappa^{n+1}(\theta,t)}d\theta
  \end{equation*}
or
  \begin{equation*}
  \lambda(t)=\frac{L^{2}(t)-2\pi A(t)}{\pi L^{2}(t)}\int_{X (\cdot, t)}\kappa^{1-n}ds=
  \frac{L^{2}(t)-2\pi A(t)}{\pi L^{2}(t)}\int^{2\pi}_{0}\frac{1}{\kappa^{n}(\theta,t)}d\theta.
  \end{equation*}
 Since the radius of curvature $\rho(\theta,t)=\frac{1}{\kappa(\theta,t)}$, we can get the following evolution equation
\begin{equation}\label{rt}
\frac{\rho(\theta,t)}{\partial t}=(\rho^{n})_{\theta\theta}(\theta,t)+\rho^{n}(\theta,t)-
\lambda(t),\quad(\theta,t)\in S^{1}\times[0,T)
\end{equation}
with $\rho(\theta,0)=\rho_{0}(\theta)=\frac{1}{\kappa_{0}(\theta)}>0$. To obtain a better-looking evolution equation, we let $\nu(\theta,t)=\rho^{n}(\theta,t)$ and get
\begin{equation}\label{nt}
\frac{\partial\nu}{\partial t}=n\nu^{p}(\nu_{\theta\theta}+\nu-\lambda(t)),\quad p=1-\frac{1}{n},
\end{equation}
where
\begin{equation}
\lambda(t)=\frac{L(t)}{2L^{2}(t)-4\pi A(t)}\int^{2\pi}_{0}\rho^{n+1}(\theta,t)d\theta=
\frac{L(t)}{2L^{2}(t)-4\pi A(t)}\int^{2\pi}_{0}\nu^{1+\frac{1}{n}}(\theta,t)d\theta,
\end{equation}
or
\begin{equation}
\lambda(t)=\frac{L^{2}(t)-2\pi A(t)}{\pi L^{2}(t)}\int^{2\pi}_{0}\rho^{n}(\theta,t)d\theta=
\frac{L^{2}(t)-2\pi A(t)}{\pi L^{2}(t)}\int^{2\pi}_{0}\nu(\theta,t)d\theta.
\end{equation}

The asymptotic behavior of the flow solution $X(\cdot,t)$ is determined by the asymptotic behavior of the scalar solution \eqref{kt}, or \eqref{rt},or \eqref{nt}.

\begin{remark}
For convenience, we can write $\frac{dL(t)}{dt}$ and $\frac{dA(t)}{dt}$ in the following form
\begin{equation}\label{lan}
\frac{dL(t)}{dt}=\int^{2\pi}_{0}\nu d\theta-2\pi\lambda(t)\quad and\quad
\frac{dA(t)}{dt}=\int^{2\pi}_{0}\nu^{1+\frac{1}{n}} d\theta-L\lambda(t).
\end{equation}
\end{remark}
\subsection{Uniform convexity of the evolving curve}
In this subsection, We shall use several lemmas to deduce the lower bound of curvature $\kappa$.
\begin{lemma}\label{est}
Under the flow \eqref{flow1}(or \eqref{flow2}) with $n\ge1$, there holds the estimate
\begin{equation}
\max\limits_{S^{1}\times[0,t]}\Phi\le\max\bigg\{\max\limits_{S^{1}\times[0,t]}\nu^{2},\max\limits_{S^{1}\times\{0\}}\Phi\bigg\},
\end{equation}
where $\nu=\kappa^{-n}$ and $\Phi=\nu^{2}+\nu^{2}_{\theta}$.
\end{lemma}
\begin{proof}
 Fix a $t>0$ and suppose that at $(\theta_{0},t_{0})\in S^{1}\times[0,t]$ we have $\Phi(\theta_{0},t_{0})=\max_{S^{1}\times[0,t]}\Phi$
(where $t_{0}>0$, otherwise we are done). Then we claim that $\nu_{\theta}(\theta_{0},t_{0})=0$ and thus the conclusion is proven. Indeed, since at the maximum of $\Phi$ we have $\nu_{\theta}(\nu_{\theta\theta}+\nu)=0$, if $\nu_{\theta}(\theta_{0},t_{0})\ne0$, then we have $\nu_{\theta\theta}+\nu=0$ at $(\theta_{0},t_{0})$. Using that, a direct calculation shows that at $(\theta_{0},t_{0})$, we have
\begin{equation*}
\frac{\partial\Phi}{\partial t}=n\nu^{n}\Phi_{\theta\theta}-2np\lambda(t)\nu^{p-1}(\nu_{\theta})^{2}-2n\lambda(t)\nu^{p+1}<0,
\end{equation*}
which contradicts with the fact that $\frac{\partial\Phi}{\partial t}(\theta_{0},t_{0})\ge0$. This completes the proof.
\end{proof}

\begin{lemma}\label{new}
Under the flow \eqref{flow1} (or \eqref{flow2}) with $n\ge1$. Assume that for some $(\theta_{0},t_{0})\in S^{1}\times(0,T)$,
\begin{equation*}
\nu(\theta_{0},t_{0})=\max\limits_{S^{1}\times[0,t_{0}]}\nu(\theta,t).
\end{equation*}
Then for any small $\epsilon>0$, there exists a number $\eta>0$ depending only on $\epsilon$, such that for all $\theta\in(\theta_{0}-\eta,\theta_{0}+\eta)$
\begin{equation*}
(1-\epsilon)\nu(\theta_{0},t_{0})\le\nu(\theta,t_{0})+\epsilon\sqrt{c},
\end{equation*}
where $c$ is the constant only depending on the initial curve.
\end{lemma}
\begin{proof}
By Lemma \ref{est},we have
\begin{align*}
\nu(\theta_{0},t_{0})&=\nu(\theta,t_{0})+\int^{\theta_{0}}_{\theta}\nu_{\theta}(\theta,t_{0})d\theta\\
&\le\nu(\theta,t_{0})+\lvert\theta_{0}-\theta\lvert\max\limits_{\theta\in S^{1}}\lvert\nu_{\theta}(\theta,t_{0})\lvert\\
&\le\nu(\theta,t_{0})+\lvert\theta_{0}-\theta\lvert\sqrt{\max\limits_{S^{1}\times[0,t_{0}]}\nu^{2}(\theta,t)+c}\\
&=\nu(\theta,t_{0})+\lvert\theta_{0}-\theta\lvert\sqrt{\nu^{2}(\theta_{0},t_{0})+c}\\
&\le\nu(\theta,t_{0})+\eta\nu(\theta_{0},t_{0})+\eta\sqrt{c}.
\end{align*}
Choose $\eta=\epsilon$. This finishes the proof.
\end{proof}

\begin{lemma}\label{lok}(Lower bound of the curvature)
Assume $X(u, t):S^{1}\times[0,T)\rightarrow\mathbb{R}^{2}$ is a smooth convex solution of the flow \eqref{flow1} (or \eqref{flow2}). Then there exists a constant $c_{1}>0$ ,which is independent of time, such that
\begin{equation*}
\kappa(\theta,t)\ge c_{1}>0,\quad\forall(\theta,t)\in S^{1}\times[0,T).
\end{equation*}
\end{lemma}
\begin{proof}
At first we claim that there exists a time-independent constant $c_{2}>0$ such that
\begin{equation}\label{nuu}
\max\limits_{S^{1}\times[0,T)}\nu(\theta,t)\le c_{2}.
\end{equation}
If \eqref{nuu} does not hold, then we can find a sequence $\{t_{i}\}^{\infty}_{i=1}\rightarrow T$ such that $\max_{S^{1}}(\theta,t_{i})$ goes to infinity as $t\rightarrow\infty$. By combining Lemma \ref{new}, one can get $L(t_{i})=\int^{2\pi}_{0}\nu^{\frac{1}{n}}(\theta,t_{i})d\theta$ goes to infinity as $t\rightarrow\infty$, which contradicts the fact that $L(t)\le L(0)$. Thus, the claim holds. Then by \eqref{nuu} we have
\begin{equation}\label{ke}
\big(\frac{1}{\min\limits_{S^{1}\times[0,T)}\kappa(\theta,t)}\big)^{n}=\max\limits_{S^{1}\times[0,T)}\nu(\theta,t)\le c_{2},
\end{equation}
which implies that the curvature $\kappa(\theta,t)$ has a time-independent positive lower bound. This finishes the proof.
\end{proof}

\begin{corollary}
(Bounds on the nonlocal term $\lambda(t)$)
Assume $X(u, t):S^{1}\times[0,T)\rightarrow\mathbb{R}^{2}$ is a smooth convex solution of the flow \eqref{flow1} (or \eqref{flow2}). Then there exist two positive  constants $c_{3}$ and $c_{4}$, independent of time,such that
\begin{equation}\label{laj}
c_{3}\le\lambda(t)\le c_{4},\quad\forall t\in[0,T).
\end{equation}
\end{corollary}
\begin{proof}
By \eqref{hold}, we have
\begin{equation}\label{ine1}
\int^{2\pi}_{0}\frac{1}{\kappa^{n+1}}d\theta\ge\frac{L^{n+1}}{(2\pi)^{n}}.
\end{equation}
Combining \eqref{la} and \eqref{nuu}, we cam easily get \eqref{laj}.
\end{proof}

\subsection{Long time existence}
The curvature estimate established so far and the parabolic regularity theory implies the
following:
\begin{lemma}\label{longtime}(Long time existence of the flow)
Assume $n\ge1$ and $X_{0}(u),u\in S^{1}$, is a smooth convex closed curve. Consider the flow \eqref{flow1} (or \eqref{flow2}) and assume that the curvature $\kappa$ of the evolving curve will not blow up in any finite time during the evolution, then it has a unique smooth convex solution $X(u,t):S^{1}\times[0,\infty)\rightarrow\mathbb{R}^{2}$ defined for all time $t\in[0,\infty)$.
\end{lemma}
\begin{proof}
By parabolic theory, there exists a unique smooth convex solution $X(u,t):S^{1}\times[0,T)\rightarrow\mathbb{R}^{2}$ defined for short time $T>0$. On the domain $S^{1}\times[0,T)$, by Lemma \ref{lok} and the assumption, the curvature $\kappa$ has uniform positive upper and lower bounds. Parabolic regularity then implies that all space-time derivatives of $\kappa$ are uniformly bounded on $S^{1}\times[0,T)$. Thus, the evolving convex closed curve $X(\cdot,t)$ converges in $C^{\infty}$ to a smooth convex closed curve $X(\cdot,T)$ as $t\to T$ and we can use $X(\cdot,T)$ as a new initial curve to continue this flow.For more details of argument, see Lemma 3.2 in \cite{gpt1}.
\end{proof}

\subsection{Convergence of the flow under the long time existence assumption}
\begin{lemma}
(Convergence of the isoperimetric difference)
Assume $n\ge1$ and $X(u, t):S^{1}\times[0,\infty)\rightarrow\mathbb{R}^{2}$ is a smooth convex solution of the flow \eqref{flow1} (or \eqref{flow2}). Then the isoperimetric difference $L^{2}(t)-4\pi A(t)$ decreases and decays to zero exponentially as $t\to\infty$.
\end{lemma}
\begin{proof}
By \eqref{lan}, we have
\begin{align*}
\frac{d}{dt}\big(L^{2}(t)-4\pi A(t)\big)&=2L(t)\big(\int^{2\pi}_{0}\nu d\theta-2\pi\lambda(t)\big)-4\pi\big(\int^{2\pi}_{0}\nu^{1+\frac{1}{n}} d\theta-L\lambda(t)\big)\\
&=2L(t)\int^{2\pi}_{0}\frac{1}{\kappa^{n}(\theta,t)}d\theta-4\pi\int^{2\pi}_{0}\frac{1}{\kappa^{n+1}(\theta,t)}d\theta\le0,
\end{align*}
where the inequality is due to the H$\ddot{o}$lder inequality
$$\int^{2\pi}_{0}\frac{1}{\kappa}d\theta\int^{2\pi}_{0}\frac{1}{\kappa^{n}}d\theta\le\int^{2\pi}_{0}d\theta
\int^{2\pi}_{0}\frac{1}{\kappa^{n+1}}d\theta.$$
By \eqref{LT}, we obtain
\begin{equation}\label{lattt}
\frac{d}{dt}\big(L^{2}(t)-4\pi A(t)\big)\le-\frac{2}{L(t)}\big(L^{2}(t)-4\pi A(t)\big)\int^{2\pi}_{0}\frac{1}{\kappa^{n}}d\theta.
\end{equation}
By \eqref{lajie} and \eqref{ine1}, we have
\begin{equation*}
\int^{2\pi}_{0}\frac{1}{\kappa^{n}}d\theta\ge\frac{L^{n}}{(2\pi)^{n-1}}\ge\frac{(\sqrt{4\pi A(0)})^{n}}{(2\pi)^{n-1}},
\end{equation*}
which, together with \eqref{lattt}, implies
\begin{equation}
\frac{d}{dt}\big(L^{2}(t)-4\pi A(t)\big)\le-\frac{2(\sqrt{4\pi A(0)})^{n}}{L(0)(2\pi)^{n-1}}\big(L^{2}(t)-4\pi A(t)\big).
\end{equation}
This finishes the proof.
\end{proof}

\begin{lemma}\label{ntjj}
Assume $n\ge1$ and $X(u, t):S^{1}\times[0,\infty)\rightarrow\mathbb{R}^{2}$ is a smooth convex solution of the flow \eqref{flow1} (or \eqref{flow2}). Then there exist positive constants $D_{i}$(i=1,2,3\dots), independent of time, such that
\begin{equation}\label{ntj}
\bigg\lvert\frac{\partial^{i}\nu}{\partial\theta^{i}}(\theta,t)\bigg\lvert\le D_{i},\quad\forall (\theta,t)\in S^{1}\times[0,\infty).
\end{equation}
\end{lemma}
\begin{proof}
We will prove the case $i=1$ only. The proof for $i>1$ is similar and can be argued by mathematical induction(see \cite{LT}). By \eqref{nuu}, we can find two positive constants $c_{5}$ and $c_{6}$ such that
\begin{equation}\label{nuji}
0<c_{5}\le\nu(\theta,t)\le c_{6}.
\end{equation}
Let $U=\nu_{\theta}+\xi\nu^{2}$, where $\xi$ is a constant to be chosen later on. By \eqref{nt}, we have
\begin{align*}
\frac{\partial U}{\partial t}=&(\nu_{\theta})_{t}+2\xi\nu\nu_{t}\\
=&n\nu^{p}\nu_{\theta\theta\theta}+np\nu^{p-1}\nu_{\theta}\nu_{\theta\theta}
+np\nu^{p-1}(\nu-\lambda)+n\nu^{p}\nu_{\theta}
+2n\xi\nu^{p+1}\nu_{\theta\theta}+2n\xi\nu^{p+1}(\nu-\lambda)\\
=&n\nu^{p}(U_{\theta\theta}-2\xi\nu\nu_{\theta\theta}-2\xi\nu_{\theta}^{2})+np\nu^{p-1}\nu_{\theta}(U_{\theta}
-2\xi\nu\nu_{\theta})+n\nu^{p-1}\nu_{\theta}\big(p(\nu-\lambda)+\nu\big)\\
&+2n\xi\nu^{p+1}\nu_{\theta\theta}+2n\xi\nu^{p+1}(\nu-\lambda)\\
=&n\nu^{p}U_{\theta\theta}+np\nu^{p-1}\nu_{\theta}U_{\theta}-2n\xi\nu^{p}(p+1)\nu_{\theta}^{2}+n\nu^{p-1}
\big(p(\nu-\lambda)+\nu\big)\nu_{\theta}+2n\xi\nu^{p+1}(\nu-\lambda).
\end{align*}
Since $\nu^{2}, \nu^{p}$ and $\lambda$ are bounded quantities, if $U$ becomes sufficiently large (either negatively or positively), it must be due to the term $\nu_{\theta}$. The coefficient of $\nu_{\theta}^{2}$ is
\begin{equation*}
-2n\xi\nu^{p}(p+1)=2n\xi\nu^{p}(\frac{1}{n}-2)=\xi\nu^{p}(2-4n),\quad p=1-\frac{1}{n}.
\end{equation*}
Choose $\xi=-1$, then $U=\nu_{\theta}-\nu^{2}$ and $\xi\nu^{p}(2-4n)=\nu^{p}(4n-2)$ is strictly positive and if
$U(\theta^{\star},t)=U_{\min}(t)$ becomes sufficiently large (negatively) at the minimum point $\theta=\theta^{\star}$, we have
\begin{equation*}
\frac{\partial U}{\partial t}\ge\nu^{p}(4n-2)\nu_{\theta}^{2}+n\nu^{p-1}
\big(p(\nu-\lambda)+\nu\big)\nu_{\theta}-2n\nu^{p+1}(\nu-\lambda)>0\,\,at\,\,(\theta^{\star},t).
\end{equation*}
The minimum principle implies that the function $U=\nu_{\theta}-\nu^{2}$ has a negative time-independent lower bound and so is the function $\nu_{\theta}$. Similarly, if we choose $\xi=1$, then $U=\nu_{\theta}+\nu^{2}$ and $\xi\nu^{p}(2-4n)=\nu^{p}(2-4n)$ is strictly negative and if
$U(\theta^{\star},t)=U_{\max}(t)$ becomes sufficiently large (positively) at the maximum point $\theta=\theta^{\star}$, we have
\begin{equation*}
\frac{\partial U}{\partial t}\le\nu^{p}(2-4n)\nu_{\theta}^{2}+n\nu^{p-1}
\big(p(\nu-\lambda)+\nu\big)\nu_{\theta}+2n\nu^{p+1}(\nu-\lambda)>0\,\,at\,\,(\theta^{\star},t).
\end{equation*}
The maximum principle implies that the function $U=\nu_{\theta}-\nu^{2}$ has a positive time-independent upper bound and so is the function $\nu_{\theta}$. This finishes the proof.
\end{proof}

\begin{lemma}
Assume $n\ge1$ and $X(u, t):S^{1}\times[0,\infty)\rightarrow\mathbb{R}^{2}$ is a smooth convex solution of the flow \eqref{flow1} (or \eqref{flow2}). Then there exists a positive constant $c_{7}$, independent of time, such that
\begin{equation}\label{latj}
\big\lvert\lambda^{'}(t)\big\lvert\le c_{7},\quad\forall t\in[0,\infty).
\end{equation}
\end{lemma}
\begin{proof}
For $\lambda(t)$ given by \eqref{flow1} we have
\begin{align*}
\big\lvert\lambda^{'}(t)\big\lvert=&\bigg\lvert\frac{d}{dt}\big(\frac{L}{2L^{2}-4\pi A}\big)\int^{2\pi}_{0}\nu^{1+\frac{1}{n}}d\theta+\frac{L}{2L^{2}-4\pi A}\frac{d}{dt}\bigg(\int^{2\pi}_{0}\nu^{1+\frac{1}{n}}d\theta\bigg)\bigg\lvert\\
=&\bigg\lvert\frac{L_{t}(2L^{2}-4\pi A)-(4LL_{t}-4\pi A_{t})L}{(2L^{2}-4\pi A)^{2}}\int^{2\pi}_{0}\nu^{1+\frac{1}{n}}d\theta+\frac{(1+\frac{1}{n})L}{2L^{2}-4\pi A}
\int^{2\pi}_{0}\nu^{\frac{1}{n}}\nu_{t}d\theta\big\lvert\\
=&\bigg\lvert\frac{(\int^{2\pi}_{0}\nu d\theta-2\pi\lambda)(2L^{2}-4\pi A)-\big(4L(\int^{2\pi}_{0}\nu d\theta-2\pi\lambda)-4\pi(\int^{2\pi}_{0}\nu^{1+\frac{1}{n}} d\theta-L\lambda)\big)L}{(2L^{2}-4\pi A)^{2}}\\
&\int^{2\pi}_{0}\nu^{1+\frac{1}{n}}d\theta+\frac{(1+\frac{1}{n})L}{2L^{2}-4\pi A}
\int^{2\pi}_{0}n\nu(\nu_{\theta\theta}+\nu-\lambda)d\theta\bigg\lvert\\
=&\bigg\lvert\frac{4\pi L\int^{2\pi}_{0}\nu^{1+\frac{1}{n}}d\theta-(2L^{2}+4\pi A)\int^{2\pi}_{0}\nu d\theta
+8\pi^{2}A\lambda}{(2L^{2}-4\pi A)^{2}}\int^{2\pi}_{0}\nu^{1+\frac{1}{n}}d\theta\\
&+\frac{(1+n)L}{2L^{2}-4\pi A}\int^{2\pi}_{0}(\nu^{2}-\nu^{2}_{\theta}-\lambda\nu)d\theta\bigg\lvert
\end{align*}
Combining \eqref{la}, \eqref{nuu}, \eqref{laj} and \eqref{ntj}, we can obtain \eqref{latj}. For $\lambda(t)$ given by \eqref{flow2}, the proof is similar.
\end{proof}

\begin{lemma}\label{latji}
Assume $n\ge1$ and $X(u, t):S^{1}\times[0,\infty)\rightarrow\mathbb{R}^{2}$ is a smooth convex solution of the flow \eqref{flow1} (or \eqref{flow2}). Then we have
\begin{equation}
\frac{dL}{dt}(t)\to 0 \quad as \quad t\to\infty,
\end{equation}
and
\begin{equation}
 \frac{dA}{dt}(t)\to 0 \quad as \quad t\to\infty.
\end{equation}
\end{lemma}
\begin{proof}
The idea is to estimate $L^{''}(t)$ and $A^{''}(t)$. We will use the following useful result from calculus: Let $f(t)\ge0$ be a differential function on $[0,\infty)$ with $\int^{\infty}_{0}f(t)dt\le\infty$. If there exists a constant $C>0$ such that $\lvert f^{'}(t)\lvert\le C$ on $[0,\infty)$, then we must have $f(t)\to0$ as $t\to\infty$.

Now let $f_{1}(t)=-L^{'}(t), f_{2}(t)=A^{'}(t)$. By Lemma \ref{la} and \ref{lajie}, we know that $f_{1}(t)\ge0, f_{2}(t)\ge0$
on $[0,\infty)$ with
$$\int^{\infty}_{0}f_{1}(t)dt=L(0)-L(\infty)<\infty\quad and\quad \int^{\infty}_{0}f_{2}(t)dt=A(\infty)-A(0)<\infty.$$
By \eqref{lan}, we have
\begin{align*}
\big\lvert L^{''}(t)\big\lvert&=\bigg\lvert\int^{2\pi}_{0}\nu_{t} d\theta-2\pi\lambda^{'}(t)\bigg\lvert
=\bigg\lvert\int^{2\pi}_{0}n\nu^{p}(\nu_{\theta\theta}+\nu-\lambda(t)) d\theta-2\pi\lambda^{'}(t)\bigg\lvert\\
&=\bigg\lvert n\int^{2\pi}_{0}\big(\nu^{p+1}-p\nu^{p-1}\nu_{\theta}^{2}-\nu^{p}\lambda(t)\big) d\theta-2\pi\lambda^{'}(t)\bigg\lvert
\end{align*}
and
\begin{align*}
\big\lvert A^{''}(t)\big\lvert&=\bigg\lvert L^{'}(t)\lambda(t)+L(t)\lambda^{'}(t)-(1+\frac{1}{n})\int^{2\pi}_{0}\nu^{\frac{1}{n}}\nu_{t} d\theta\bigg\lvert\\
&=\bigg\lvert L^{'}(t)\lambda(t)+L(t)\lambda^{'}(t)-(n+1)\int^{2\pi}_{0}(\nu^{2}-\nu^{2}_{\theta}-\lambda(t)\nu)d\theta\bigg\lvert.
\end{align*}
By \eqref{la}, \eqref{nuu}, \eqref{laj} , \eqref{ntj} and\eqref{latj}, there exist two positive constants $c_{8}$ and $c_{9}$, independent of time, such that
\begin{equation}
\big\lvert L^{''}(t)\big\lvert\le c_{8}\quad and\quad\big\lvert A^{''}(t)\big\lvert\le c_{9}
\end{equation}
The above two estimates imply $\frac{dL}{dt}(t)\to 0$ and $\frac{dA}{dt}(t)\to 0$ as $t\to\infty.$ This finishes the proof.
\end{proof}

As a consequence of Lemma \ref{latji}, we have
\begin{lemma}\label{lain}
Assume $n\ge1$ and $X(u, t):S^{1}\times[0,\infty)\rightarrow\mathbb{R}^{2}$ is a smooth convex solution of the flow \eqref{flow1} (or \eqref{flow2}). Then we have
\begin{equation}
\lim\limits_{t\to\infty}\bigg\Vert\nu(\theta,t)-\bigg(\frac{L(\infty)}{2\pi}\bigg)^{n}\bigg\Vert_{C^{0}(S^{1})}=0.
\end{equation}
\end{lemma}
\begin{proof}
By \eqref{ntj} and \eqref{nuji}, for any sequence $t_{i}\to\infty$ the family of functions $\{\nu(\theta,t_{i})\}^{\infty}_{i=1}$ is equicontinuous on $S^{1}$. Hence there exists a subsequence, still denoted as $t_{i}$, such that $\nu(\theta,t_{i})$ converges uniformly to $\nu_{\infty}(\theta)$ on $S^{1}$, where $\nu_{\infty}(\theta)$ is some nonnegative continuous bounded function on $S^{1}$.

For $\lambda(t)$ given by \eqref{flow1}, we have
\begin{equation*}
\frac{dL}{dt}(t)=\int^{2\pi}_{0}\nu d\theta-\frac{\pi L(t)}{L^{2}(t)-2\pi A(t)}\int^{2\pi}_{0}\nu^{1+\frac{1}{n}} d\theta\to 0 \quad as \quad t\to\infty,
\end{equation*}
and
\begin{equation*}
 \frac{dA}{dt}(t)=\int^{2\pi}_{0}\nu^{1+\frac{1}{n}} d\theta-\frac{L^{2}(t)}{2L^{2}(t)-4\pi A(t)}\int^{2\pi}_{0}\nu^{1+\frac{1}{n}} d\theta\to 0 \quad as \quad t\to\infty.
\end{equation*}
Then we get
\begin{equation}\label{ll}
\int^{2\pi}_{0}\nu_{\infty}(\theta) d\theta=\frac{\pi L(\infty)}{L^{2}(\infty)-2\pi A(\infty)}\int^{2\pi}_{0}\nu^{1+\frac{1}{n}}_{\infty}(\theta) d\theta,
\end{equation}
and
\begin{equation}
\int^{2\pi}_{0}\nu^{1+\frac{1}{n}}_{\infty}(\theta) d\theta=\frac{L^{2}(\infty)}{2L^{2}(\infty)-4\pi A(\infty)}\int^{2\pi}_{0}\nu^{1+\frac{1}{n}}_{\infty}(\theta) d\theta.
\end{equation}
Hence $L^{2}(\infty)=4\pi A(\infty)$ and \eqref{ll} becomes
\begin{equation}\label{lll}
\int^{2\pi}_{0}\nu_{\infty}(\theta) d\theta=\frac{2\pi}{L(\infty)}\int^{2\pi}_{0}\nu^{1+\frac{1}{n}}_{\infty}(\theta) d\theta.
\end{equation}
By
$$L(\infty)=\lim\limits_{i\to\infty}L(t_{i})=\lim\limits_{i\to\infty}\int^{2\pi}_{0}
\nu^{\frac{1}{n}}(\theta,t_{i}) d\theta=\int^{2\pi}_{0}\nu^{\frac{1}{n}}_{\infty}(\theta)d\theta$$
we can rewrite \eqref{lll} as
\begin{equation*}
\int^{2\pi}_{0}\nu^{\frac{1}{n}}_{\infty}(\theta)d\theta\int^{2\pi}_{0}\nu_{\infty}(\theta) d\theta=
\int^{2\pi}_{0} d\theta\int^{2\pi}_{0}\nu^{1+\frac{1}{n}}_{\infty}(\theta) d\theta,
\end{equation*}
i.e., the following iterated integral vanishes
\begin{equation}\label{ab}
\frac{1}{2}\int^{2\pi}_{0}\int^{2\pi}_{0}\bigg[\big(\nu^{\frac{1}{n}}_{\infty}(x)-\nu^{\frac{1}{n}}_{\infty}(y)\big)
(\nu_{\infty}(x)-\nu_{\infty}(y))\bigg]dxdy=0.
\end{equation}
As the integrand in \eqref{ab} is nonnegative everywhere on $S^{1}\times S^{1}$, it must be equal to zero everywhere, which implies $\nu_{\infty}(x)=\nu_{\infty}(y)$ for all $x,y\in S^{1}$ and so $\nu_{\infty}(\theta)$
 must be a constant function on $S^{1}$ with $\nu_{\infty}(\theta)=(L(\infty)/2\pi)^{n}$. Because every sequence has a subsequence converging uniformly to the same limit, we must have $\nu(\theta,t)\to(L(\infty)/2\pi)^{n}$ as $t\to\infty$. For $\lambda(t)$ given by \eqref{flow2}, the proof is similar.
\end{proof}

As a consequence of the Arzela-Ascoli theorem, Lemma \ref{ntjj} and Lemma \ref{lain}, we immediately have the following smooth convergence.
\begin{lemma}\label{nuin} ($C^{\infty}$ convergence of $\nu$)
Assume $n\ge1$ and $X(u, t):S^{1}\times[0,\infty)\rightarrow\mathbb{R}^{2}$ is a smooth convex solution of the flow \eqref{flow1}(or \eqref{flow2}). Then we have
\begin{equation}
\lim\limits_{t\to\infty}\bigg\Vert\nu(\theta,t)-\bigg(\frac{L(\infty)}{2\pi}\bigg)^{n}\bigg\Vert_{C^{m}(S^{1})}=0,
\quad \forall m=0,1,2,3\cdots
\end{equation}
\end{lemma}

This Lemma says that the evolving curve $X(\cdot,t)$ converges to a circle smoothly in the sense that its radius of curvature (or curvature) converges smoothly to a constant. Conceivably the evolving curve $X(\cdot,t)$ may escape to infinity or oscillate indefinitely. To show that this will not happen, we will prove that $X(\cdot,t)$ has limit as $t\to\infty$.
\begin{remark}
Lemma \ref{nuin} tells us the the limiting curve $X_{\infty}(\cdot)$ is a fixed circle with radius $L(\infty)/2\pi$. The center $(a,b)\in\mathbb{R}^{2}$ of  $X_{\infty}(\cdot)$ is given by
\begin{equation}
(a,b)=\lim\limits_{t\to\infty}\frac{1}{2\pi}\int^{2\pi}_{0}X(\theta,t)d\theta=
\lim\limits_{t\to\infty}\frac{1}{\pi}\int^{2\pi}_{0}s(\theta,t)(\cos\theta,\sin\theta)d\theta,
\end{equation}
where $s(\theta,t)$ is the support function of the curve $X(\cdot,t)$.
\end{remark}

\begin{lemma}\label{xyz}
Assume $n\ge1$ and $X(u, t):S^{1}\times[0,\infty)\rightarrow\mathbb{R}^{2}$ is a smooth convex solution of the flow \eqref{flow1}(or \eqref{flow2}. Then there exist positive constants $c_{10},c_{11}$, both are independent of time, so that
\begin{equation}
\bigg\lvert\frac{\partial X}{\partial t}(u,t)\bigg\lvert\le c_{10}e^{-c_{11}t},\quad \forall(u,t)\in S^{1}\times[0.\infty).
\end{equation}
In particular, we have
\begin{equation}
\lim\limits_{t\to\infty}X(u,t)=X_{0}(u)+\int^{\infty}_{0}X_{t}(u,t)dt,\quad\forall u\in S^{1},
\end{equation}
where, for each $u\in S^{1}$, the integral $\int^{\infty}_{0}X_{t}(u,t)dt$ converges and $X(u,t)$ has a limit
$X_{\infty}(u)$ as $t\to\infty$.
\end{lemma}
\begin{proof}
By \eqref{nt}, we have
\begin{align}\label{1}
\frac{d}{dt}\int^{2\pi}_{0}\frac{1}{2}(\nu_{\theta})^{2}d\theta&=\int^{2\pi}_{0}\nu_{\theta}\nu_{\theta t}d\theta\notag\\
&=\int^{2\pi}_{0}n\nu^{p}\nu_{\theta}\nu_{\theta\theta\theta}d\theta+\int^{2\pi}_{0}np\nu^{p-1}\nu_{\theta}^{2}
\nu_{\theta\theta}d\theta+\int^{2\pi}_{0}n\nu^{p-1}\big(p(\nu-\lambda)+\nu\big)\nu_{\theta}^{2}d\theta\notag\\
&=-\int^{2\pi}_{0}(n\nu^{p}\nu_{\theta})_{\theta}\nu_{\theta\theta}d\theta+\int^{2\pi}_{0}np\nu^{p-1}\nu_{\theta}^{2}
\nu_{\theta\theta}d\theta+\int^{2\pi}_{0}n\nu^{p-1}(p(\nu-\lambda)+\nu)\nu_{\theta}^{2}d\theta\notag\\
&=-\int^{2\pi}_{0}n\nu^{p}(\nu_{\theta\theta})^{2}d\theta+\int^{2\pi}_{0}n\nu^{p-1}(p(\nu-\lambda)+\nu)\nu_{\theta}^{2}d\theta\notag\\
&=-n\int^{2\pi}_{0}\nu^{1-\frac{1}{n}}(\nu_{\theta\theta})^{2}d\theta+n\int^{2\pi}_{0}\nu^{1-\frac{1}{n}}
(\nu_{\theta})^{2}d\theta+(n-1)\int^{2\pi}_{0}(\nu-\lambda)\nu^{-\frac{1}{n}}(\nu_{\theta})^{2}d\theta.
\end{align}
Since $\nu=\rho^{n}$, we have
\begin{align}
\int^{2\pi}_{0}(\nu_{\theta\theta})^{2}d\theta&=\int^{2\pi}_{0}\bigg(
n\rho^{n-1}\rho_{\theta\theta}+n(n-1)\rho^{n-2}(\rho_{\theta})^{2}\bigg)^{2}d\theta\notag\\
&=n^{2}\int^{2\pi}_{0}\rho^{2n-2}(\rho_{\theta\theta})^{2}+\int^{2\pi}_{0}
\bigg(2n^{2}(n-1)\rho^{2n-3}\rho_{\theta\theta}+n^{2}(n-1)^{2}\rho^{2n-4}(\rho_{\theta})^{2}\bigg)(\rho_{\theta})^{2}d\theta.\notag
\end{align}
By Lemma \ref{nuin} and Wirtinger inequality, for any small $\epsilon>0$, if $t$ is large enough, we can obtain
\begin{align}\label{11}
\int^{2\pi}_{0}(\nu_{\theta\theta})^{2}d\theta&\ge n^{2}\bigg(\big(\frac{L(\infty)}{2\pi}\big)^{2n-2}-\epsilon\bigg)
\int^{2\pi}_{0}(\rho_{\theta\theta})^{2}d\theta-\epsilon\int^{2\pi}_{0}(\rho_{\theta})^{2}d\theta\notag\\
&\ge4n^{2}\bigg(\big(\frac{L(\infty)}{2\pi}\big)^{2n-2}-\epsilon\bigg)
\int^{2\pi}_{0}(\rho_{\theta})^{2}d\theta-\epsilon\int^{2\pi}_{0}(\rho_{\theta})^{2}d\theta\notag\\
&=\bigg[4n^{2}\bigg(\big(\frac{L(\infty)}{2\pi}\big)^{2n-2}-\epsilon\bigg)-\epsilon\bigg]\int^{2\pi}_{0}(\rho_{\theta})^{2}d\theta.
\end{align}
On the other hand, we also have for large $t$ the estimate
\begin{equation}\label{12}
\int^{2\pi}_{0}(\nu_{\theta})^{2}d\theta=\int^{2\pi}_{0}(n\rho^{n-1}\rho_{\theta})^{2}d\theta=
n^{2}\int^{2\pi}_{0}\rho^{2n-2}(\rho_{\theta})^{2}d\theta\le n^{2}\bigg(\big(\frac{L(\infty)}{2\pi}\big)^{2n-2}+\epsilon\bigg)\int^{2\pi}_{0}(\rho_{\theta})^{2}d\theta.
\end{equation}
By \eqref{11} and \eqref{12}, one obtains
\begin{equation}\label{13}
\int^{2\pi}_{0}(\nu_{\theta\theta})^{2}d\theta\ge\frac{4n^{2}\bigg(\big(\frac{L(\infty)}{2\pi}\big)^{2n-2}-\epsilon\bigg)
-\epsilon}{n^{2}\bigg(\big(\frac{L(\infty)}{2\pi}\big)^{2n-2}+\epsilon\bigg)}\int^{2\pi}_{0}(\nu_{\theta})^{2}d\theta
\ge(4-\sigma)\int^{2\pi}_{0}(\nu_{\theta})^{2}d\theta
\end{equation}
for some number $\sigma>0$ satisfying $\sigma\to0$ as $\epsilon\to0$. Therefore we obtain for $t$ large
\begin{align}
n\int^{2\pi}_{0}\nu^{1-\frac{1}{n}}(\nu_{\theta\theta})^{2}d\theta&\ge n
\bigg(\big(\frac{L(\infty)}{2\pi}\big)^{1-\frac{1}{n}}-\epsilon\bigg)\int^{2\pi}_{0}(\nu_{\theta\theta})^{2}d\theta\notag\\
&\ge n\bigg(\big(\frac{L(\infty)}{2\pi}\big)^{1-\frac{1}{n}}-\epsilon\bigg)(4-\sigma)\int^{2\pi}_{0}(\nu_{\theta})^{2}d\theta.
\end{align}
By combining \eqref{1}, we conclude for large $t$ the estimate
\begin{align}
\frac{d}{dt}\int^{2\pi}_{0}\frac{1}{2}(\nu_{\theta})^{2}d\theta\le&
-n\bigg(\big(\frac{L(\infty)}{2\pi}\big)^{1-\frac{1}{n}}-\epsilon\bigg)(4-\sigma)\int^{2\pi}_{0}(\nu_{\theta})^{2}d\theta\notag\\
&+n\bigg(\big(\frac{L(\infty)}{2\pi}\big)^{1-\frac{1}{n}}+\epsilon\bigg)\int^{2\pi}_{0}(\nu_{\theta})^{2}d\theta
+\epsilon\int^{2\pi}_{0}(\nu_{\theta})^{2}d\theta\notag\\
&\le-c_{12}\int^{2\pi}_{0}(\nu_{\theta})^{2}d\theta,
\end{align}
for some constant $c_{12}>0$, which is independent of $t,\epsilon,\sigma$ as long as $t$ is large enough and  $\epsilon,\sigma$ are small enough. Thus we have
\begin{equation}
\int^{2\pi}_{0}(\nu_{\theta})^{2}d\theta\le c_{13}e^{-2c_{12}t},\quad\forall t\in[0,\infty),
\end{equation}
where $c_{13}$ is positive constant depending only on the initial data $X_{0}$. By combining Sobolev's inequality, we can estimate the speed of the flow \eqref{flow1}
\begin{align}\label{14}
\Vert X_{t}(u,t)\Vert_{C^{0}(S^{1})}&=\bigg\Vert\nu-\frac{L}{2L^{2}-4\pi A}\int^{2\pi}_{0}\nu^{1+\frac{1}{n}}d\theta\bigg\Vert_{C^{0}(S^{1})}\notag\\
&=\bigg\Vert\nu-\frac{L}{2L^{2}-4\pi A}\int^{2\pi}_{0}\nu\kappa^{-1}d\theta\bigg\Vert_{C^{0}(S^{1})}\notag\\
&\le\bigg\Vert\nu-\frac{c_{14}L}{2L^{2}-4\pi A}\int^{2\pi}_{0}\nu d\theta\bigg\Vert_{C^{0}(S^{1})}\notag\\
&\le\sqrt{2\pi}\bigg(\int^{2\pi}_{0}(\nu-\beta(t))^{2}d\theta\bigg)^{\frac{1}{2}}
+\frac{1}{\sqrt{2\pi}}\bigg(\int^{2\pi}_{0}(\nu_{\theta})^{2}d\theta\bigg)^{\frac{1}{2}}\notag\\
&\le(\sqrt{2\pi}+\frac{1}{\sqrt{2\pi}})\bigg(\int^{2\pi}_{0}(\nu_{\theta})^{2}d\theta\bigg)^{\frac{1}{2}}\notag\\
&\le(\sqrt{2\pi}+\frac{1}{\sqrt{2\pi}})\sqrt{c_{13}}e^{-c_{12}t},
\end{align}
where $c_{14}$ is the upper bound of $\kappa$ and $\beta(t)=\frac{c_{14}L}{2L^{2}-4\pi A}\int^{2\pi}_{0}\nu d\theta$. \eqref{14} tells us that the speed of the flow \eqref{flow1} decays exponentially. The evolving curve $X(\cdot,t)$ has a limit $X_{\infty}(\cdot)$ as $t\to\infty$ due to the convergence of the integral
\begin{equation}
\lim\limits_{t\to\infty}X(u,t)=X_{0}(u)+\int^{\infty}_{0}X_{t}(u,t)dt,\quad\forall u\in S^{1}.
\end{equation}
For $\lambda(t)$ given by \eqref{flow2}, the proof is similar.
\end{proof}

The proof of Theorem \ref{them} is now complete due to Lemma \ref{llaa}, Lemma \ref{lok}, Lemma \ref{longtime}, Lemma \ref{nuin} and Lemma \ref{xyz}.

\textbf{Acknowledgements}
This work was partially supported by Science and Technology Commission
of Shanghai Municipality (No. 22DZ2229014). The research is supported by Shanghai Key Laboratory of PMMP.

\textbf{Data availability}
Data sharing  not applicable to this article as no datasets were generated or analysed during the current study.

\textbf{Conflict of interest}
The author has no conflicts of interest to declare that are relevant to the content of this article.


\end{document}